\documentclass{article}
\usepackage{amsmath,amssymb,amsthm}
\usepackage{ascmac}
\usepackage{color}
\theoremstyle{plain}
\newtheorem{theorem}{Theorem}
\newtheorem{proposition}{Proposition}[section]

\numberwithin{equation}{section}

\newtheorem{corollary}[proposition]{Corollary}
\newtheorem{problem}{Problem}

\newtheorem{Remark}{Remark}[section]

\newtheorem{example}{Example}[section]
\newtheorem{lemma}{Lemma}[section]
\title{Geometric convexity of an operator mean}
\author{Shuhei Wada 
\thanks{E-mail: wada@j.kisarazu.ac.jp}
\\ 
Department of Information and Computer Engineering, \\ 
	National Institute of Technology, Kisarazu College
}
\date{}
\begin{document}
\maketitle
\begin{abstract}
Let $\sigma$ be an operator mean in the sense of Kubo and Ando. 
If the representation function $f_\sigma$ of $\sigma$ satisfies 
$$f_\sigma (t)^p\le f_\sigma(t^p) \text{ for all } p>1,$$
then $\sigma$ is called a pmi mean.
Our main interest is the class of pmi means (denoted by $PMI$).
To study $PMI$, 
the operator mean $\sigma$, wherein
$$f_\sigma(\sqrt{xy})\le \sqrt{f_\sigma (x)f_\sigma (y)}\quad (x,y>0)$$
is considered in this paper. The set of such means
(denoted by $GCV$) includes certain significant examples and is contained in $PMI$. 
The main result presented in this paper is that $GCV$ is a proper subset of $PMI$. 
In addition, we investigate certain operator-mean classes, which contain $PMI$. 
\end{abstract}

\section{Introduction}
Let ${\mathcal H}$ be a Hilbert space with an inner product $\langle \cdot ~|~ \cdot \rangle $. 
 A bounded linear operator $A$ on ${\mathcal H}$ is said to be positive (denoted by $A\ge 0$) if 
$\langle Ax~|~x\rangle \ge 0$ for all $x \in {\mathcal H}$.
We denote the set of positive operators on ${\mathcal H}$ by $B({\mathcal H})_+$.
If an operator $A\in B({\mathcal H})_+$ is invertible, we denote $A>0$. 

 A continuous real function $f$ from $(0,\infty)$
 is said to be operator monotone on $(0,\infty)$, if 
the inequality $A \ge B>0$ implies $f(A) \ge f(B)$. 

An operator monotone function $f$ on $(0, \infty)$ is called normal, if $f(1)=1$. 
In this paper, $OM_+^1$ denotes the set of normalized  
operator monotone functions on $(0,\infty)$ into itself. 

In \cite{K-A}, Kubo and Ando provide
the following axiom for operator means. 
A binary operation $\sigma$ among $B(\mathcal {H})_+$ 
is called an operator mean, if it satisfies the following: 
\begin{enumerate}
\item[{\rm (i)}] $A \le C, B\le D \Rightarrow A \sigma B \le C \sigma D,$
\item[{\rm (ii)}] $T^*(A \sigma B)T \le (T^*AT)\sigma (T^*BT),$
\item[{\rm (iii)}] $A_n \downarrow A,B_n \downarrow B \Rightarrow 
A_n\sigma B_n \downarrow A\sigma B$,
\item[{\rm (iv)}] $1\sigma 1 =1$.
\end{enumerate}
\noindent
If $f$ is in $OM_+^1$, then 
the binary operation $\sigma_{f}$ on $B(\mathcal {H})_+$ defined by 
$$A\sigma_{f} B =\lim_{\epsilon \downarrow 0}
A_{\epsilon}^{1 \over 2}
f(A_{\epsilon}^{-{1 \over 2}}B_{\epsilon}^{1 \over 2}
A_{\epsilon}^{-{1 \over 2}})
A_{\epsilon}^{1 \over 2}
$$ is an operator mean, 
where $A_{\epsilon}=A+\epsilon 1$ and $B_{\epsilon}=B+\epsilon 1$. 
Kubo and Ando show that 
the function $f \mapsto \sigma_f$ is an order isomorphism from $OM_+^1$ onto 
the set of operator means \cite{K-A}. In this paper, we call $\sigma_f$ an operator mean 
corresponding to $f$ and at times identify $\sigma_f$ as $f$. 

The following 
theorem is referred to as the Ando-Hiai inequality (\cite{A-H}). 
$A,B>0, A\#_\alpha B \ge I \Rightarrow A^p\#_\alpha B^p \ge I\quad (p\ge 1),$
where $\alpha\in [0,1]$ and $\#_\alpha$ is an operator mean corresponding to a power function 
$t\mapsto t^\alpha$. 
In \cite{wada}, it is shown that the generalized inequality
$$A,B>0, A\sigma_f B \ge I \Rightarrow A^p\sigma_f B^p \ge I\quad (p\ge 1)$$
holds if and only if $f$ is power monotone increasing (pmi for short), i.e., 
\begin{equation}\label{pmi}f(t)^p\le f(t^p)\quad (p\ge 1,t>0).
\end{equation}
Our main interest is the class of pmi means (denoted by $PMI$).
To study $PMI$, in this paper,
we consider an operator mean $\sigma_f$, wherein 
\begin{equation}\label{gcv}
f(\sqrt{xy})\le \sqrt{f(x)f(y)}\quad (x,y>0)
\end{equation}
holds. A positive valued function with (\ref{gcv}) is called geometrically convex or multiplicatively convex 
(\cite{B-H}, \cite{Niclescu}); hence, 
we denote the set of functions $f\in OM_+^1$ with (\ref{gcv}) 
by $GCV$.

In {\color{black}Section 3}, we present some of the basic properties of $GCV$ and 
its adjoint (denoted by $GCC$). From this argument, we conclude that 
several significant $PMI$ means are contained in $GCV$.

It is conjectured that $GCV$ is a proper subset of $PMI$, 
i.e., 
$$GCV\subsetneq PMI.$$
{\color{black}In Section 4, 
we characterize a pmi mean and a gcv mean by using 
Hansen's integral representaion of an operator mean \cite{hansen}.  
Using this, we prove the above conjecture, which is our main result. }

In Section 5, we consider $PMI_\infty$ defined by 
$$PMI_\infty:=\{f\in OM_+^1~|~f(t)\ge t^{f'(1)} \}$$
and prove that $PMI$ is a proper subset of this class. 

Combining the above arguments,
we finally obtain the following relationships:
Let $\sigma$ be an operator mean and $f_\sigma$ be the representation function of $\sigma$.
Consider the following statements:   

(I) $f_\sigma$ is geometrically convex. 

(II) $A,B>0, A\sigma B \ge I \Rightarrow A^p\sigma B^p \ge I\quad (p\ge 1).$

 (III) $\sigma\ge \#_\alpha$ for certain $\alpha\in [0,1]$. 
\hfill\break
Then 

(1) I implies II; II implies III, 

(2) III does not imply II; II does not imply I.

\section{Geodesic mean}
As per the theory of Kubo and Ando \cite{K-A}, the set $OM_+^1$ of 
normalized positive operator monotone functions on $(0,\infty)$ is 
identified with the set of operator means. Hence, the following classes  
$${\mathcal PMI}:=\{f\in OM_+^1~|~f(t)^r\le f(t^r) \quad (\forall r>1)\}$$
and 
$${\mathcal PMD}:=\{f\in OM_+^1~|~f(t)^r\ge f(t^r) \quad (\forall r>1)\}$$ 
can be viewed as subsets of the set of operator means. 
The function $f\in PMI$ (resp. $f\in PMD$)  
is referred to as a
pmi (resp. pmd) mean. 
As stated in \cite{wada}, 
for any probability measure $p$ on $[0,1]$, the function 
\begin{equation}\label{gm}
x\mapsto \int_0^1 x^\alpha dp(\alpha)
\end{equation} 
is in $PMI$. Such a function $f$ is called a
geodesic mean  
and the set of geodesic means is denoted by $GM$. 
Several examples of 
a pmi mean can be obtained by using the fact that $GM\subset PMI$. 

Although there are a number of functions belonging to PMI, 
it is not easy to show the pmi property (\ref{pmi}) of a certain operator mean 
because the verification of condition (\ref{pmi}) or (\ref{gm})
requires considerable calculation. Bourin and Hiai \cite{B-H} mention that 
a positive operator monotone function $f$ on $[0,\infty)$ belongs to $GM$ if and only if
${{d^n}\over {dt}^n} f(e^t)\ge 0$ 
for all $n\ge 0$. 
Thus we need to determine a technique 
for obtaining pmi means and evaluate this technique. 

\section{Geometrically convex mean}
\subsection{Definitions and basic properties}

A positive function $f$ on $(0,\infty)$ is called 
geometrically convex (resp. geometrically concave), if 
$$f(\sqrt{xy})\le \sqrt{f(x)f(y)} \quad (resp. \quad f(\sqrt{xy})\ge \sqrt{f(x)f(y)})$$
holds for all $x,y >0$. Let $gcv$ (resp. $gcc$) be the set of monotone increasing continuous functions that 
are geometrically convex (resp. \ geometrically concave ) on $(0,\infty)$. 
We also define $GCV$ (resp. $GCC$) as follows :
$$GCV:=\{f\in OM_+^1~|~ f \in gcv \}\quad (resp. \ \ GCC:=\{ f\in OM_+^1~|~  f \in gcc \}).$$

As stated in \cite{B-H}, the convexity of 
the function $t\mapsto \log f(e^t)$ is 
a necessary and sufficient condition for $f\in OM_+^1$ to be in $GCV$. 
 Using this, we have some inclusions among the subclasses of $OM_+^1$. 
The second inclusion in the following is proved in \cite{Fujii-Yamazaki}. 
\begin{proposition}\label{key_lemma} 
 $GM\subseteq GCV \subseteq PMI$. 
\end{proposition}
\begin{lemma}\label{basic}
Let $f$ and $g$ be in $gcv$ and let $h\in GCV$. 
Then we have the following. 
\begin{itemize}
\item[(1)] $f\cdot g$ and $f^\alpha$ are in $gcv$ for all $\alpha >0$;
\item[(2)] The function $t\mapsto (f \sigma_h g )(t)\left(:=f (t)\cdot h\left({{g(t)} \over {f(t)}} \right)\right)$ is in $gcv$; 
\item[(3)] If $f$ is bijective, the inverse function of $f$ is in $gcc$.
\end{itemize}
\end{lemma}
\begin{proof}
The geometric convexity of $f\cdot g$ and $f^\alpha$ is immediate. 

From the definition of $GCV$, inequalities 
\begin{align*}
(f\sigma_h g)(\sqrt{xy}) 
&= f(\sqrt{xy}) \sigma_h g(\sqrt{xy}) \\
&\le \sqrt{f(x)f(y)} \sigma_h \sqrt{g(x)g(y)} \\
&=\sqrt{f(x)f(y)} h\left( \sqrt{ \left({{g(x)}\over {f(x)}}\right) \left({{g(y)}\over {f(y)}}\right) } \right) \\
&\le \sqrt{f(x)f(y)} 
\sqrt{
h \left({{g(x)}\over {f(x)}}\right) 
h \left({{g(y)}\over {f(y)}}\right) } \\
&=\sqrt{
(f\sigma_h g)(x)\ (f\sigma_h g)(y) 
}
\end{align*}
hold for all $x,y>0$. This implies (2).

Let $f^{(-1)}$ be the inverse function of $f$. Then for every $x,y>0$, 
$$f\left( \sqrt{f^{(-1)}(x) f^{(-1)}(y)}\right) \le
\sqrt{xy}, 
$$ signifying that $f^{(-1)}$ is geometrically concave.
\end{proof}
Using (2) of the above lemma, for $f,g \in gcv$, 
the weighted arithmetic mean of $f$ and $g$ is in $gcv$, which signifies that 
$gcv$ is a convex set. This implies that $GCV$ is a convex set. 
\begin{proof}[Proof of Proposition \ref{key_lemma}]
From the above lemma, $GCV$ is a convex set and 
has a power function $x^\alpha$ $(0\le \alpha \le 1)$, which implies 
the first inclusion $GM\subseteq GCV$. 

Next, we prove the second inclusion. As stated above, $f\in GCV$ if and only if function 
$t\mapsto F(t):=\log f(e^t)$ is convex on $(-\infty,\infty)$. Thus 
$$F((1-\alpha)t + \alpha s)\le (1-\alpha)F(t) +\alpha F(s)$$
holds for all $\alpha\in [0,1]$. Considering $t=0$, $F(\alpha s)\le \alpha F(s)$, 
which implies that $f(x^\alpha)\le f(x)^\alpha$ for all $x>0$. 
\end{proof}

\hfill\break

Recall that $f(t) \mapsto f^*(t)(:={1\over {f(1/t)}})$ is an idempotent mapping on $OM_+^1$   
and $PMD=PMI^*\left(:=\{f^*~|~f\in PMII\}\right)$. The following is obtained. 
\begin{corollary}
$$GM^*\subseteq GCC\subseteq PMD.$$
\end{corollary}
\begin{proof}
From the preceding result, $$GM^*\subseteq GCV^*=GCC\subseteq PMI^*=PMD.$$
\end{proof}
{\Remark From the above lemma, the set $gcv$ is closed under the sum,  i.e.,  
$f_1,f_2\in gcv \Rightarrow f_1+f_2\in gcv. $
However, the same does not hold for $gcc$. For example, 
${{2t}\over {t+1}}$ and $t^2$ are in gcc and 
${{2t}\over {t+1}}+t^2$ is not in gcc.
}

\hfill\break

Before closing this section, we note that 
there are some counterexamples for $GM=GCV$.

{\example\label{b_p} (\cite{B-H})
Let $p\in [-1,1]$ and $b_p(t):=\left( {t^p+1}\over 2\right)^{1/p}$.
Then $b_p \in GCV\backslash GM$ if and only if 
$p\in (0,1)\backslash \{{1\over n} ~|~n\in {\mathbb N}\}$.
}
{\example   (\cite{B-H})
Let $\alpha\in [-1,2]$ and $u_\alpha(t):={{\alpha-1}\over \alpha} 
{{t^\alpha-1}\over {t^{\alpha-1}-1}}$. Then 
$u_\alpha \in GCV\backslash GM$ if and only if  
$\alpha\in [1/2,2]\backslash \{1, {{m+1}\over m} , {m\over {m+1}}~|~
m\in {\mathbb N}\}$.
}

\subsection{Functions in $GCV$}
In this section, we present a few examples of a function in $GCV\left( \subseteq PMI \right)$. 
We first consider the function $u_\alpha$ defined in the previous section. 
The geometric convexity of $u_\alpha$ is characterized as follows \cite{B-H} :
$$u_\alpha\in GCV\ 
\text{(resp.}\ u_\alpha\in GCC\text{)} \iff 1/2\le \alpha\le 2\ 
\text{(resp.} -1\le \alpha\le 1/2\text{)}.$$

The  function $u_\alpha$ is generalized as $u_{a,b}$ defined by
$$u_{a,b}(t):={{b}\over a}{{t^a-1}\over {t^b-1}}\quad (a,b \in [-2,2],(a,b)\not=(0,0)),$$
where ${{t^a-1}\over a}$ is defined as $\log t$, when $a=0$.
In \cite[Example 3.4(1)]{nagisa-wada}, 
it is proved that $u_{a,b}\in OM_+^1$ if and only if $(a,b)$ is in $\Gamma$, 
where 
\begin{align*}
\Gamma:=&\{ (a,b) ~|~ 0<a-b\le 1, 2\ge a\ge -1, -2\le b\le 1\} \\
&\cup \left( [0,1]\times [-1,0]\right)\backslash \{(0,0)\} \\
&\cup \{(a,a)~|~a\not=0 \}.
\end{align*}
\begin{proposition}\label{ex_power}
$u_{a,b} \in GCV$ 
if and only if 
$(a,b)\in \Gamma$ and $|a|\ge |b|$.
\end{proposition}
\begin{proof}
We first consider the case, where $ab=0$. 
If $a=0$ and $b\not=0$, then 
$u_{a,b}(t)= {b\over {t^b-1}} \log t$ and 
$$ {{d^2}\over{dx^2}}\log u_{a,b} (e^x)
=
{-1\over {x^2}}+b^2 ( e^{bx} + e^{-bx} -2 )^{-1} < 0
$$
for $x\not=0$. Thus we obtain 
$$ {{d^2}\over{dx^2}} \log u_{b,a}(e^x) 
=
- {{d^2}\over{dx^2}} \log u_{a,b}(e^x) \ge 0,$$
which implies the desired result. 

We next consider the case, where $a\ne 0, b\ne 0$. Then there exists 
$\alpha\in {\mathbb R}$ such that 
$$u_{a,b}(t)={{|b|} \over {|a|}}{{t^{|a|}-1}\over {t^{|b|} -1}}t^\alpha$$
and 
$${{d^2}\over{dx^2}}\log u_{a,b} (e^x)= {{(|a|x)^2 \psi(|a| x)-(|b|x)^2\psi(|b| x)}\over {x^2}},$$
where $\psi (x) =-( e^x + e^{-x} -2 )^{-1}$.
The function $\Psi (y) :=y^2\psi(y)$ is a negative valued function on $(-\infty,\infty)\backslash\{0\}$ 
and 
$$\Psi(-y)=\Psi(y),\quad \Psi(x) < \Psi(y)$$
for $0<x<y$. Thus 
$${{d^2}\over{dx^2}}\log u_{a,b} (e^x) \ge 0 \iff
\Psi(|a|x) \ge \Psi(|b|x) \text{ for all }x\not=0
\iff 
 |a|\ge |b|.$$
\end{proof}

From 
$${{d^2}\over{dx^2}}\log u_{a,b}^*(e^x)=
{{d^2}\over{dx^2}}\log u_{b,a}(e^x)=
- {{d^2}\over{dx^2}}\log u_{a,b}^*(e^x)
,$$
the following is obtained.
\begin{corollary}
$u_{a,b} \in GCC$ 
if and only if 
$(a,b)\in \Gamma$ and $|a|\le |b|$.
\end{corollary}
From the above results, we have 
$u_{a,b}\in GCV \cup GCC$ for all $(a,b)\in \Gamma$, which 
{\color{black}implies} a condition for $u_{a,b}$ to be in $PMI$. 

\begin{corollary} Let $(a,b)\in \Gamma$.  Then 
the following are equivalent: 
\begin{itemize}
\item[(1)] $|a|\ge |b|$\ (resp. $|a|\le |b|$); 
\item[(2)] $u_{a,b} \in GCV$\ (resp. $u_{a,b} \in GCC$     );
\item[(3)] $u_{a,b} \in PMI$\ (resp.  $u_{a,b} \in PMD$     ).
\end{itemize}
\end{corollary}

\hfill\break

The Stolarsky mean is defined as
$$S_\alpha(s,t): =\left({{s^\alpha-t^\alpha} \over {\alpha(s-t)}}\right)^{1\over {\alpha-1}}\text{for }\alpha \in [-2,2]\backslash \{0,1\},$$
$S_0(s,t):=\lim_{\alpha\rightarrow 0}S_\alpha(s,t)$ and 
$S_1(s,t):=\lim_{\alpha\rightarrow 1}S_\alpha (s,t)$. 
It is known that $S_\alpha(1,t)$ is operator monotone, if $-2\le \alpha \le 2$ \cite{Naka}. 
Using $$\log S_\alpha(1,e^x)={1\over {\alpha-1}}\log u_{\alpha,1}(e^x)
,$$
we have a condition for $S_\alpha(1,t)$ to be in $GCV$. 
\begin{corollary}
$S_\alpha(1,t) \in GCV$ \ (resp. $S_\alpha (1,t)\in GCC$) if and only if  $\alpha \in [-1,2]$ 
(resp. $\alpha\in [-2,-1]$).
\end{corollary}
\begin{proof}
By simple calculation, we have 
$${{d^2}\over{dx^2}}\log S_1 (1,e^x) ={{d^2}\over{dx^2}}\left( {{xe^x}\over {e^x-1}}
\right) \ge 0.$$
Thus $S_1(1,t)\in GCV$. 

We next consider the case, where $\alpha\not=1$. 
By Proposition \ref{ex_power}, 
$${{d^2}\over{dx^2}}\log u_{a,b}(e^x) \le 0 \quad (resp.\ \ge 0) 
\iff |a| \le |b|\quad (resp.\ |a|\ge |b|).$$ Thus 
$${{d^2}\over{dx^2}}\log S_\alpha (1,e^x)=
{1\over {|1-\alpha |}}{{d^2}\over{dx^2}}\log u_{|\alpha |, 1}(e^x) \ge 0 \quad (1<\alpha \le 2) ,$$
 $${{d^2}\over{dx^2}}\log S_\alpha (1,e^x)=
{1\over {|1-\alpha |}} {{d^2}\over{dx^2}}\log u_{1,|\alpha |}(e^x) \ge 0\quad (-1\le \alpha <1) $$
and 
$${{d^2}\over{dx^2}}\log S_\alpha (1,e^x)=
{1\over {|1-\alpha |}} {{d^2}\over{dx^2}}\log u_{1,|\alpha |}(e^x) \le 0\quad (-2\le \alpha \le -1) .
$$
\end{proof}

\subsection{Inverses}
In \cite{Ando}, Ando proves that for every $f\in OM_+^1$, 
the function $t\mapsto tf(t)$ has the inverse function $(tf)^{(-1)}$ which is in $OM_+^1$, i.e.,
$$f\in OM_+^1 \Rightarrow (t f)^{(-1)}\in OM_+^1.$$
In this section, we investigate this result with respect to the theory 
of geometrically convex functions. 

Let $P$ be the set of nonnegative operator monotone functions on $[0,\infty)$ and 
$$P^{-1}:=\{h\in P~|~ h([0,\infty))=[0,\infty),\ \ h^{(-1)}\in P\}.$$
In \cite{Uchiyama2}, Uchiyama proves the product formula
\begin{equation}\label{product formula}
P\cdot P^{-1}\subset P^{-1}.
\end{equation}
Using this, Ando's result stated above can be extended as 
$$f\in OM_+^1 \Rightarrow (t^\alpha f)^{(-1)}\in OM_+^1\quad (\alpha\ge 1).$$

The following is immediate from the above argument.
\begin{proposition}\label{inverse1}
Let $\alpha\ge 1$ and $f\in OM_+^1$. Then 
$$f\in GCV\iff (t^\alpha f)^{(-1)}\in GCC.$$ 
\end{proposition}
\begin{proof}
Assume $f\in GCV$. Then it is evident that 
\begin{equation*}
{{d^2}\over {dt^2}}\log (t^\alpha f)( e^t )= {{d^2}\over {dt^2}}\log f(e^t) \ge 0,
\end{equation*}
which implies that $t^\alpha f \in gcv$ and $(t^\alpha f)^{(-1)} \in gcc$. 
The operator monotonicity of $(t^\alpha f)^{(-1)}$ comes from (\ref{product formula}). 

Conversely, if $(t^\alpha f)^{(-1)}\in GCC$, 
$(t^\alpha f)$ is in $gcv$. Thus 
$ {{d^2}\over {dt^2}}\log f(e^t) ={{d^2}\over {dt^2}}\log (t^\alpha f)( e^t )\ge 0$. 
\end{proof}
From $\left( (t^\alpha f)^{(-1)}\right)^*=(t^\alpha f^*)^{(-1)}\text{ and } GCC^*=GCV$, 
the preceding proposition can be rewritten as follows: 
\begin{corollary}
Let $\alpha\ge 1$ and $f\in OM_+^1$. Then 
$$f\in GCC\iff (t^\alpha f)^{(-1)}\in GCV.$$ 
\end{corollary}
\hfill\break

We next consider a function $u(t)$ defined by 
 \begin{equation}\label{polynomial}
u(t):= \beta \prod_{i=1}^n (t+a_i)^{\gamma_i}, 
\end{equation}
where $0=a_1<a_2<\cdots a_n$, $1\le \gamma_1$, $0< \gamma_i$ and $0< \beta$. 
Uchiyama shows that $u$ is in $P^{-1}$ and this result can be derived using the above product formula (\cite{Uchiyama},\cite{Uchiyama2}).
Additionally, we show the following.
\begin{proposition}
If $f\in GCV$ and $u(1)=1$, then $(u\cdot f)^{(-1)} \in GCC$.
\end{proposition}
\begin{proof}
From Lemma \ref{basic}, we have $(u\cdot f) \in gcv$ and 
$(u\cdot f)^{(-1)} \in gcc$.
The operator monotonicity of $(u\cdot f)^{(-1)}$
comes from (\ref{product formula}).
\end{proof}
As the constant function $1$ is in $GCV$, the following is evident: 
\begin{corollary} 
If $u(1)=1$, then $u^{(-1)}\in GCC$.
\end{corollary}

{
\example
For $\alpha\in (0,1)$, a function 
$u(t):=t(1-\alpha+ \alpha t)$ has the inverse  
$u^{(-1)}(s)={{\alpha-1 +\sqrt{(1-\alpha)^2+4s\alpha }}\over {2\alpha}}$  
and $u^{(-1)}$ is in $GCC$.
}

\section{Main results}
In this section, we present an integral representation of an element of  $GCV$. 
In \cite{hansen}, Hansen considers a class of real valued continuous functions defined as
$${\mathcal E}:=\{   
F~|~ F:{\mathbb R}\rightarrow {\mathbb R}\text{ is continuous and } 
e^A \le e^B \Rightarrow e^{F(A)} \le e^{F(B)} 
\},$$
and proves the following :
\begin{itemize}
\item[(1)] 
{\color{black}A function $F:{\mathbb R}\rightarrow {\mathbb R}$ is in ${\mathcal E}$}
 if and only if
there exists $\beta\in {\mathbb R}$ and a measurable function,
$h : (-\infty ,0]\rightarrow [0,1]$ such that
$$F(x)=\beta+\int_{-\infty}^0 \left( {1\over {\lambda -e^x}}- {\lambda\over {\lambda^2+1}}
\right) h(\lambda) d\lambda,$$
where $d\lambda$ is the Lebesgue measure on $(-\infty,0]$; 
\item[(2)] the preceding measurable function $h$ is uniquely determined by $F$;
\item[(3)] the function $F \mapsto\exp F (\log t)$ is a bijection from ${\mathcal E}$ onto 
$P$.
\end{itemize}
From this result, for $f\in OM_+^1$, 
 $F(x)\left(:=\log f (e^x)\right)$ can be expressed as 
$$F(x)=\beta+\int_{-\infty}^0 \left( {1\over {\lambda -e^x}}- {\lambda\over {\lambda^2+1}}
\right) h(\lambda) d\lambda.$$
As $F(0)=0$, 
$$\beta=\int_{-\infty}^0 \left( {-1\over {\lambda -1}}+ {\lambda\over {\lambda^2+1}}
\right) h(\lambda) d\lambda.$$
Thus 
$$f(t)=\exp \int_{-\infty}^0 \left( {1\over {\lambda -t}}- {1 \over {\lambda-1}}
\right) h(\lambda) d\lambda. $$
Using this, we obtain the following : 
\begin{proposition}\label{GCV_int}
Let $f\in OM_+^1$ and let $h$ be a measurable function determined using the above method. 
 Then $f\in PMI$ if and only if
\begin{equation}\label{PMI_char}
\int_{-\infty}^0 
\left( {1\over {\lambda -t^r}}
-
 {r\over {\lambda -t}}+ {{r-1} \over {\lambda-1}}
\right) h(\lambda) d\lambda
\ge 0
\end{equation}
for all $t>0$ and $r\ge 1$. 
Moreover, $f\in GCV$ if and only if 
\begin{equation}\label{GCV_char}
\int_{-\infty}^0 
\left( 
{{\lambda +t}\over {(\lambda -t)^3}}
\right) h(\lambda) d\lambda \ge 0
\end{equation}
for all $t>0$.
\end{proposition}
\begin{proof}
From the above argument, 
$$F(x)=\log f(e^x)=\int_{-\infty}^0 \left( {1\over {\lambda -e^x}}- {1 \over {\lambda-1}}
\right) h(\lambda) d\lambda.$$ 
Hence, the condition {\color{black}$f(e^{rx})/f(e^x)^r \ge 1$} can be expressed as
$$0\le F(rx)-rF(x)= \int_{-\infty}^0 
\left( {1\over {\lambda -e^{rx}}}
-
 {r\over {\lambda -e^x}}+ {{r-1} \over {\lambda-1}}
\right) h(\lambda) d\lambda
$$ for all $x\in {\mathbb R}$ and $r\ge 1$.

From Lebesgue's dominated convergence theorem, 
the condition $${{d^2}\over {dx^2}}\log f(e^x) \ge 0$$ can be expressed as 
\begin{align*}
F''(x)&=\int_{-\infty}^0 
{{d^2}\over {dx^2}}
\left( {1\over {\lambda -e^{x}}}- {1 \over {\lambda-1}}
\right) h(\lambda) d\lambda \\
&=
\int_{-\infty}^0 
\left( 
{{e^x \lambda + e^{2x}}\over {(\lambda -e^x)^3}}
\right) h(\lambda) d\lambda \ge 0, 
\end{align*}
which implies the desired result. 
\end{proof}
{\Remark
Let $0< a<\infty$. Considering $h=I_{(-\infty,-a)} \ (resp.\ h=I_{(-a,0)})$, 
we have
$$f(t)={{a+t}\over {a+1}}\in GCV\quad (resp.\ \ 
f(t)={{(a+1)t}\over {a+t}}\in GCC
).$$
}{
\Remark Let $\alpha \in [0,1]$. Considering $ h=\alpha I_{(-\infty,0)}$, 
$$f(t)=t^\alpha \in GCV\cap GCC . $$
}

\subsection{Conjecture and theorem}
It is conjectured that $GCV$ is a proper subset of $PMI$. 
To prove this, we use the argument of the preceding section.
We set $\alpha={9\over {14}}$ and  
$h(\lambda):=\alpha I_{(-\infty,-2)} +(1-\alpha)I_{(-1,0)}$ and 
show that inequality (\ref{GCV_char}) does not hold, but (\ref{PMI_char}) holds.  
\begin{theorem}\label{GCV and PMI}
$$GCV\subsetneq PMI.$$
\end{theorem}
\begin{proof}
Let us show that (\ref{PMI_char}) holds. For $r>1$, we have 
\begin{align*}
\int_{-\infty}^0 
&\left( {1\over {\lambda -t^r}}
-
 {r\over {\lambda -t}}+ {{r-1} \over {\lambda-1}}
\right) h(\lambda) d\lambda \\ 
&=
(1-\alpha) \left( {\alpha \over {1-\alpha}} 
\log{{3^{r-1}(t^r+2)}\over {(t+2)^r}}
-
\log{{2^{r-1}(t^r+1)}\over {(t+1)^r}}
\right) . 
\end{align*}
Let us set $\beta:= {\alpha \over {1-\alpha}}$ and
$$\varphi (t):=\beta\log{{3^{r-1}(t^r+2)}\over {(t+2)^r}}
-
\log{{2^{r-1}(t^r+1)}\over {(t+1)^r}}. 
$$
Then 
$${{d\varphi}\over {dt}}={
{2\,r\,t\,(t^{r-1}-1)
\psi_{\beta,r}(t)}
\over
{(t^r+2)(t^r+1)(t+2)(t+1)}}
,
$$
where 
$
\psi_{\beta,r}(t)=
{\color{black}\left\{ 
\left(\beta-{{1}\over{2}}\right)\,t^{r+1}+\left(\beta-1
\right)\,\left(t^{r}+t\right)-(2-\beta)
\right\}}
$.
Here, $\psi_{\beta,r}$ is strictly monotone increasing  
and equation $\psi_{\beta,r}(t)=0$ has a unique solution in 
$(0,1)$. Thus 
$$\min_{t\ge 0} \varphi(t) = \min\{ \varphi(0),\varphi(1)\}=\varphi(1)= 0.$$

We next show that inequality (\ref{GCV_char}) does not hold. By simple calculation, 
$$
\int_{-\infty}^0 
\left( 
{{\lambda+t}\over {(\lambda -t)^3}}
\right) h(\lambda) d\lambda 
=
(1-\alpha) {{2}\over{(t+2)^2}} \left(  
\beta- {{(t+2)^2} \over {2(t+1)^2}}
\right), 
$$
where $\beta= {\alpha \over {1-\alpha}}\left(={9\over 5}\right)$.  
This takes a negative value, if $t$ is sufficiently small. 
\end{proof}
\begin{corollary}\label{GCV and PMI}
$$GCC\subsetneq PMD.$$
\end{corollary}

\section{Related results}
In this section, we consider some operator-mean classes containing 
$PMI$ and prove certain relationships among them. 
For $r> 1$, we define 
$${PMI}_r:=\{f\in OM_+^1~|~ f(t^r)\ge f(t)^r\}$$
and 
$$PMI_\infty := \{f\in OM_+^1~|~f(t)\ge t^\alpha \text{ for some } \alpha\in [0,1]\}.$$
 
{\color{black}We first note a property of $PMI_r$. 
It follows from \cite[Corollary 4.7]{HSW} that 
the Ando-Hiai type inequality, 
$$A,B>0,\quad A\sigma_f B \ge 1 \Rightarrow A^r \sigma_f B^r \ge 1$$
is a necessary and sufficient condition for $f\in OM_+^1$ to be in $PMI_r$. 
In addition,  from the proof of \cite[Lemma 2.1]{wada}, we have  
$$PMI = \bigcap_{x\in (1,2]}PMI_x\subset PMI_r.$$ 
}

We next consider the case, where $r=\infty$. 
Let $f\in OM_+^1$ such that $f(t)\ge t^\alpha$. 
Then,
$${{f(t)-f(1)}\over {t-1}}\ge {{t^\alpha-1}\over {t-1}}\quad (t>1)$$
and 
$${{f(t)-f(1)}\over {t-1}}\le {{t^\alpha-1}\over {t-1}}\quad (t<1),$$
which implies that  $f'(1)=\alpha$. 
Thus the definition of $PMI_\infty$ can be rewritten as follows: 
$$PMI_\infty=\{f\in OM_+^1~|~ f(t)\ge t^{f'(1)}\}.$$

As stated in \cite{Y}, the following relationship among 
$PMI_r$ and $PMI_\infty$ is known. 
\begin{proposition}(\cite{Y}) For $r>1$, 
$$PMI_r \subset PMI_\infty.$$
\end{proposition}
\begin{proof}
Let $f\in PMI_r$. It is evident from the definition that
$$f(t^{s_n})^{1/{s_n}}\le f(t)$$
for $s_n:=1/r^n$. Thus 
$$\lim_{n\rightarrow \infty} \exp( \log (f(t^{s_n})^{1/{s_n}}))
=\exp(\lim_{n\rightarrow \infty} {{\log (f(t^{s_n}))} \over {s_n}})
=t^{f'(1)}\le f(t).
$$
\end{proof}
From the above discussion, the problem whether 
$PMI$ is a proper subset of $PMI_\infty$ arises (cf. \cite{Y}).
We provide an answer to this problem. 
\begin{proposition} 
$$\bigcup_{r>1}PMI_r \subsetneq PMI_\infty. $$
\end{proposition}
\begin{proof}
We show that
$$f(t):={{(1/3)t+(2/3) t^{1/3}}\over {(1/3) + (2/3) t^{1/3}}}$$
is in $PMI_\infty\backslash PMI_r$ for all $r>1$. 

Let us show $f\in PMI_\infty$. 
As the operator monotonicity of $f$ comes from \cite{KNOW}, 
it is sufficient to show $f(t)\ge t^{f'(1)}$. Set $g(t):=f(t)-t^{1/3}$, then
$$g'(t)={{t^{1/3}\left(t^{1/3}-1\right)\,\left(t^{1/3}+1\right)\,\left(2\,t^{1/3}+1\right)\,\left(4\,t^{1/3}
 -1\right)}\over{24\,t^2+36\,t^{{{5}\over{3}}}+18\,t^{{{4
 }\over{3}}}+3\,t}}$$
and $g(0)=g(1)=0$, which implies that  $g(t)\ge 0$ and 
$f(t)\ge t^{1/3}$. 

In addition, from 
$$\lim_{t\rightarrow 0}{{f(t^r)}\over {f(t)^r}}=
\lim_{t\rightarrow 0}
{{((1/3)t^r+(2/3) t^{r/3})}\over {((1/3) + (2/3) t^{r/3})}}
{{((1/3) + (2/3) t^{1/3})^r}\over{((1/3)t+(2/3) t^{1/3})^r} }=2^{1-r} <1 , 
$$
we have $f\not\in PMI_r$.
\end{proof}
Combining all the results stated above, we obtain the following:
\begin{corollary}
Let $\sigma$ be an operator mean and $f_\sigma$ be the representation function of $\sigma$.  
Consider the statements:

(I) $f_\sigma$ is geometrically convex. 

(II) $A,B>0, A\sigma B \ge I \Rightarrow A^r\sigma B^r \ge I \text{ for all } r>1.$

(III) $A,B>0, A\sigma B \ge I \Rightarrow A^r\sigma B^r \ge I \text{ for some } r>1.$

(IV) $\sigma\ge \#_\alpha$ for some $\alpha\in [0,1]$. 
\hfill\break
Then 

(1) I implies II; II implies III; III implies IV, 

(2) IV does not imply III; II does not imply I.
\end{corollary}

\hfill\break

\hfill\break
Thus a problem 
arises.

\begin{problem}
Let $r>1$. Then,
$PMI= PMI_r ?$
\end{problem}

\hfill\break

\section*{Acknowledgement}
We gratefully acknowledge the helpful discussions with F. Hiai on several points of this paper.


\begin{thebibliography}{10}


\bibitem{Ando} 
 T. Ando ,
{\it Comparison of norms $|||f(A)-f(B)|||$ and $|||f(|A-B|)|||$ }, 
Math. Z. 197 (1988), 403--409.

\bibitem{A-H} 
T. Ando and F. Hiai ,
{\it Log majorization and complementary Golden-Thompson type inequality}, 
Linear Algebra Appl. 197 (1994), 113--131.


\bibitem{B-H} 
J.-C.Bourin and F. Hiai,
{\it
Jensen and Minkowski inequalities for operator means and 
anti--norms
}, 
Linear Algebra Appl. 456 (2014), 22--53. 

\bibitem{Fujii-Yamazaki}
J.I. Fujii and T. Yamazaki, 
{\it
Power monotonicity for a path of operator means
}, 
submitted to Scientiae Mathematicae Japonicae. 


\bibitem{hansen} 
F.Hansen, 
{\it 
Selfadjoint means and operator monotone functions
}, 
Math. Ann. 256 (1981), no.1, 29--35.


\bibitem{HSW} F. Hiai, Y. Seo and S. Wada
{\it Ando--Hiai type inequalities for multivariate
operator means},
to appear in Linear and Multilinear Algebra


\bibitem{K-A} F. Kubo and T. Ando, 
{\it Means of positive linear operators},
 Math. Ann. ~246 (1980), 205--224. 


\bibitem{KNOW} 
F. Kubo, N. Nakamura, K. Ohno and S. Wada, 
{\it Barbour path of operator monotone functions}, 
Far East J.~Math.~Sci.~ (FJMS) 57 (2)(2011), 181--192.

\bibitem{nagisa-wada}
M.Nagisa and  S. Wada
{\it 
Operator monotonicity of some functions
}, 
Linear Algebra Appl. 486 (2015), 389--408. 


\bibitem{Naka} 
Y. Nakamura,
{\it Classes of operator monotone functions and Stieltjes functions, } 
in : H. Dym, et al. (Eds.), The Gohberg Anniversary Collection, vol. II, 
Oper. Theory Adv. Appl., vol. 41, Birkh\"auser, 1989, pp. 395--404.




\bibitem{Niclescu}
C.P. Niculescu and L-E. Persson, {\it 
Convex functions and their applications. A contemporary approach }, 
CMS Books in Mathematics/Ouvrages de Math\'ematiques de la SMC. Springer, Cham, 2018

\bibitem{Uchiyama}
M. Uchiyama, {\it 
Operator monotone functions which are defined implicitly and operator inequalities
}, 
J. Funct. Anal. 175 (2000), no. 2, 330--347.

\bibitem{Uchiyama2}
M. Uchiyama, {\it 
A new majorization between functions, 
polynomials, and operator inequalities}, 
J. Funct. Anal. 231 (2006), 221--244.


\bibitem{wada} S. Wada, 
{\it Some ways of constructing Furuta-type inequalities
}, 
 Linear Algebra Appl. 457 (2014), 276--286.




\bibitem{Y} T.Yamazaki, 
{\it 
An integral representation of operator means via the power means and an application to the Ando-Hiai inequality}, 
arXiv:1803.04630.


\end{thebibliography}
\end{document}